\documentclass[12pt]{amsart}

\usepackage{psfrag}
\usepackage{color}
\usepackage{tikz}
\usetikzlibrary{matrix,arrows}
\usepackage{graphicx,graphics}
\usepackage{fullpage,amssymb,amsfonts,amsmath,amstext,amsthm,amscd,verbatim,enumerate}
\usepackage[T1]{fontenc}

\begin{document}

\newtheorem{theorem}{Theorem}[section]
\newtheorem{result}[theorem]{Result}
\newtheorem{fact}[theorem]{Fact}
\newtheorem{example}[theorem]{Example}
\newtheorem{conjecture}[theorem]{Conjecture}
\newtheorem{lemma}[theorem]{Lemma}
\newtheorem{proposition}[theorem]{Proposition}
\newtheorem{corollary}[theorem]{Corollary}
\newtheorem{facts}[theorem]{Facts}
\newtheorem{props}[theorem]{Properties}
\newtheorem*{thmA}{Theorem A}
\newtheorem{ex}[theorem]{Example}
\theoremstyle{definition}
\newtheorem{definition}[theorem]{Definition}
\newtheorem{remark}[theorem]{Remark}
\newtheorem*{defna}{Definition}

\newcommand{\notes} {\noindent \textbf{Notes.  }}
\newcommand{\note} {\noindent \textbf{Note.  }}
\newcommand{\defn} {\noindent \textbf{Definition.  }}
\newcommand{\defns} {\noindent \textbf{Definitions.  }}
\newcommand{\x}{{\bf x}}
\newcommand{\z}{{\bf z}}
\newcommand{\B}{{\bf b}}
\newcommand{\V}{{\bf v}}
\newcommand{\T}{\mathbb{T}}
\newcommand{\Z}{\mathbb{Z}}
\newcommand{\Hp}{\mathbb{H}}
\newcommand{\D}{\mathbb{D}}
\newcommand{\R}{\mathbb{R}}
\newcommand{\N}{\mathbb{N}}
\renewcommand{\B}{\mathbb{B}}
\newcommand{\C}{\mathbb{C}}
\newcommand{\ft}{\widetilde{f}}
\newcommand{\dt}{{\mathrm{det }\;}}
 \newcommand{\adj}{{\mathrm{adj}\;}}
 \newcommand{\0}{{\bf O}}
 \newcommand{\av}{\arrowvert}
 \newcommand{\zbar}{\overline{z}}
 \newcommand{\xbar}{\overline{X}}
 \newcommand{\htt}{\widetilde{h}}
\newcommand{\ty}{\mathcal{T}}
\renewcommand\Re{\operatorname{Re}}
\renewcommand\Im{\operatorname{Im}}
\newcommand{\tr}{\operatorname{Tr}}

\newcommand{\ds}{\displaystyle}
\numberwithin{equation}{section}

\renewcommand{\theenumi}{(\roman{enumi})}
\renewcommand{\labelenumi}{\theenumi}

\title{Poincar\'{e} linearizers in higher dimensions}
\subjclass[2010]{Primary 37F10; Secondary 30C65, 30D05}

\author{Alastair Fletcher}
\address{Department of Mathematical Sciences, Northern Illinois University,
DeKalb, IL 60115-2888, USA}
\email{fletcher@math.niu.edu}

\begin{abstract}
It is well-known that a holomorphic function near a repelling fixed point may be conjugated to a linear function. The function which conjugates is called a Poincar\'e linearizer and may be extended to a transcendental entire function in the plane. In this paper, we study the dynamics of a higher dimensional generalization of Poincar\'e linearizers. These arise by conjugating a uniformly quasiregular mapping in $\R^m$ near a repelling fixed point to the mapping $x\mapsto 2x$. In particular, we show that the fast escaping set of such a linearizer has a spider's web structure.
\end{abstract}

\maketitle

\section{Introduction}

\subsection{Background}

A central theme of complex dynamics is that of linearization, that is, conjugating a mapping near a fixed point to a simpler mapping. The idea is that it is then easier to see how the mapping behaves near the fixed point. For example, if $p$ is a polynomial in $\C$ with a repelling fixed point $z_0$, i.e. $p(z_0) = z_0$ and $|p'(z_0)| >1$, then there exists an entire function $L$ which satisfies $L(0) = z_0$ and
\[ p(L(z)) = L(p'(z_0)\cdot z),\]
for all $z\in \C$. Hence up to conjugation $p$ behaves like a $\C$-linear mapping near $z_0$, see for example \cite{Milnor}. The function $L$ is called a Poincar\'{e} function or a linearizer of $p$ at $z_0$. The functional equation may be iterated to obtain
\[ p^k(L(z)) = L( p'(z_0)^k \cdot z),\]
for any $k\in \N$. This indicates that the linearizer $L$ depends on the dynamical properties of $p$ as well as $z_0$ and $p'(z_0)$. 

\begin{figure}[h]
\begin{center}
\includegraphics[width=5in]{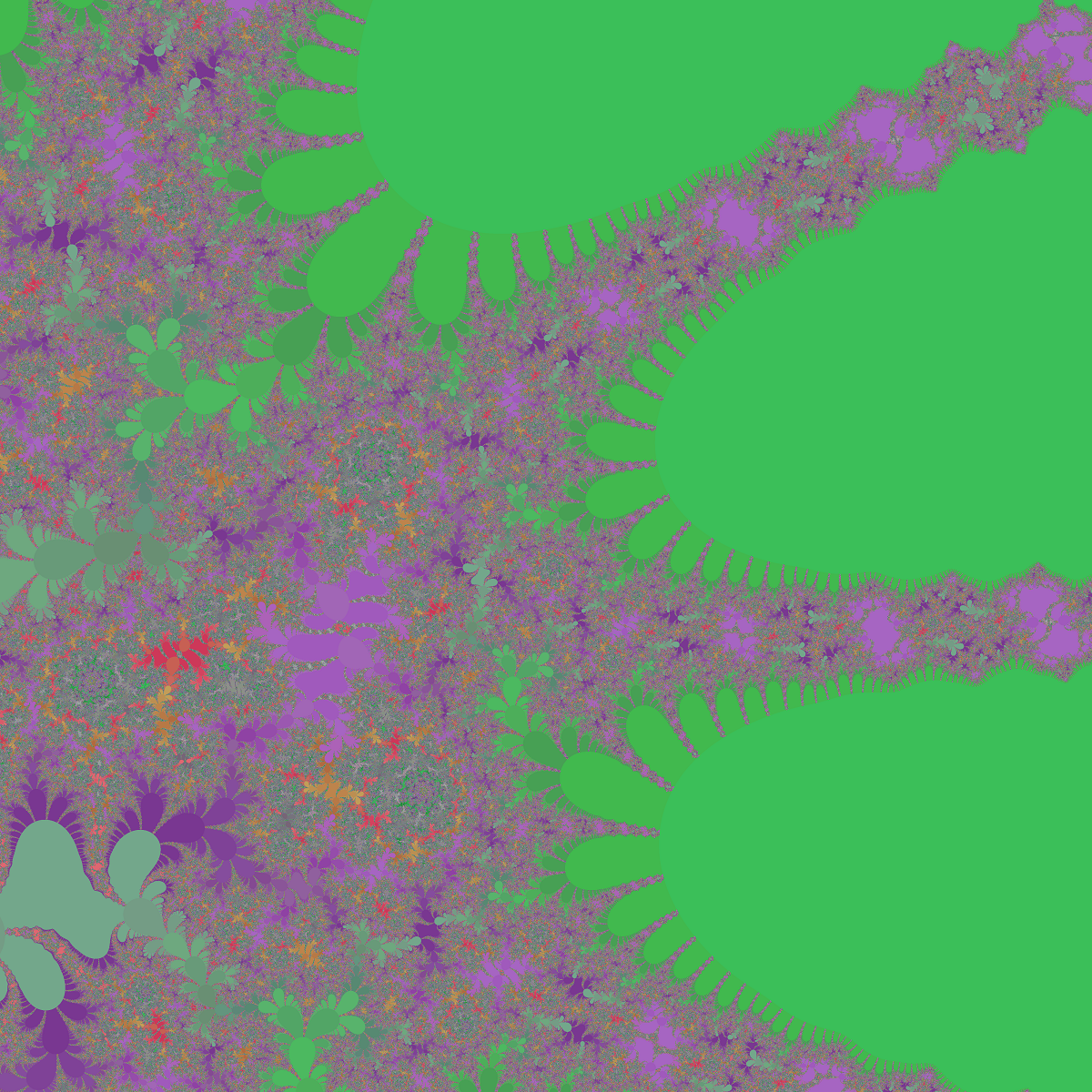}
\caption{The Julia set for the linearizer of $z^2 -0.8+0.157i$ about $z_0 = 1.528 -0.076i$ (to $3$ decimal places). The fast escaping set is a spider's web.}\label{pic1}
\end{center}
\end{figure}

The dynamics of such linearizers, and in particular the fast escaping set, were studied in \cite{MBP} by Mihaljevic-Brandt and Peter. They showed that if $c$ is not in the Mandelbrot set, then a linearizer about a fixed point of $z^2+c$ has a spider's web structure for its fast escaping set, see for example Figure \ref{pic1} which was produced by Doug Macclure.

We briefly recall the notion of the fast escaping set here.
Recall that the escaping set of a holomorphic function is defined by
\[ I(f) = \{ z\in \C : f^n(z) \to \infty \},\]
and was first studied by Eremenko \cite{E} for transcendental entire functions, and the fast escaping set is defined by
\[ A(f) = \{ z \in \C : \exists L \in \N, |f^{n+L}(z)| \geq M^n(R,f), n \in \N\}, \]
where $ M(R,f) = \max \{ |f(x) | : |x|=R \}$ is the maximum modulus and $M^n(R,f)$ denotes the iterated maximum modulus, e.g. $M^2(R,f) = M(M(R,f),f)$.
The fast escaping set was introduced by Bergweiler and Hinkkanen \cite{BH} and has been extensively studied, see for example \cite{MBP,RS,RS2012}.

Quasiregular mappings are a natural higher dimensional analogue of holomorphic functions in the plane, and so the iteration of quasiregular mappings is a natural higher dimensional counterpart to complex dynamics. An important point here is that quasiregular mappings satisfy analogues of Picard's Theorem and Montel's Theorem, both key results for complex dynamics. Rickman's monograph \cite{Rickman} is a foundational reference for quasiregular mappings. Briefly, quasiregular mappings are Sobolev mappings in $W^1_{n,loc}(\R^m)$ with a uniform bound on the distortion.
See \cite{B2} for an overview of the iteration theory of quasiregular mappings, \cite{BFLM,FN} for the escaping set of quasiregular mappings, \cite{B3,BN} for the definition of the Julia set for quasiregular mappings and \cite{BDF} for the fast escaping set of quasiregular mappings of transcendental type. Here, a quasiregular mapping of transcendental type is one for which the limit of $f(x)$ as $|x|\to \infty$ does not exist, in direct analogy with transcendental entire functions.

The aim of this article is to investigate the properties of analogues of Poincar\'{e} linearizers for quasiregular mappings. There is not a surfeit of examples of quasiregular mappings with interesting dynamics, and so an aim of this paper is to bring attention to this class of quasiregular mappings and its dynamics.

\subsection{Construction of linearizers and statement of results}

We will first construct the analogue of Poincar\'{e} functions in the quasiregular setting. Recalling how Poincar\'e linearizers arise from repelling fixed points of polynomials, the natural analogue of polynomials to this situation are uniformly quasiregular mappings (abbreviated to uqr). These are mappings $f:\R^m \to \R^m$ for which there is a uniform bound on the distortion of all the iterates. All currently known examples extend to quasiregular mappings $S^m \to S^m$, where $S^m = \R^m \cup \{ \infty \}$. These were the first quasiregular mappings whose dynamics were studied \cite{IM}.

For such mappings which are not injective, there are direct analogues of the Fatou and Julia sets and $\R^m = F(f) \cup J(f)$. The escaping set $I(f)$ is a connected neighbourhood of infinity \cite{FN} and $J(f)$ is the closure of the periodic points \cite{Siebert}. It is an open question whether $J(f)$ is the closure of the repelling periodic points, but see Theorem \ref{dense}.

Suppose $f:\R^m \to \R^m$ is a uqr mapping of polynomial type and $x_0 \in J(f)$ is a repelling periodic point. In this context, a repelling periodic point $x_0$ is one for which there exist $k\in \N$ and a neighbourhood $U \ni x_0$ such that $f^k(x_0) = x_0$, $f^k$ is injective on $U$ and $ \overline{U} \subset f^k(U)$. An immediate obstacle is that quasiregular mappings need not be differentiable everywhere. However, Hinkkanen, Martin and Mayer \cite{HMM} consider the notion of the generalized derivative (see also \cite{GMRV}). 
Without loss of generality, assume that the fixed point is $x_0=0$. For $\lambda >0$ define $f_\lambda(x) = \lambda f(x/\lambda)$.
Then the set of limit mappings
\[ \mathcal{D} f(0) =\{ \varphi \in \lim _{j\to \infty} f_{\lambda_j}, \text{ where } \lambda_j \to \infty \}\]
is called the infinitesimal space of the uqr map $f$ at $0$, and elements of the infinitesimal space are called generalized derivatives. 

\begin{remark}
If $f$ is differentiable at $x_0$, then $\mathcal{D}f(x_0)$ contains only the linear mapping $x\mapsto f'(x_0)x$.
\end{remark}

\begin{theorem}
\label{linearizer}
Let $f:\R^m \to \R^m$ be a uqr mapping of polynomial type with repelling fixed point $x_0$. Then there exists a quasiregular mapping $L:\R^m \to \R^m$ with $L(0) = x_0$ such that $f\circ L = L \circ T$, where $T(x) = 2x$.
\end{theorem}

\begin{proof}
This theorem essentially combines two well-known results.
By \cite[Theorem 6.3 (ii)]{HMM} there is a quasiregular mapping $\Psi : \R^m \to \R^m$ and a generalized derivative $\varphi$ such that $f\circ \Psi = \Psi \circ \varphi$. Here, $\varphi$ is a loxodromic uniformly quasiconformal map which fixes $0$ and infinity, and $0$ is the repelling fixed point.
Then by \cite[Theorem 1.2]{HM}, $\varphi$ is quasiconformally equivalent to $T(x) = 2x$, i.e. there exists a quasiconformal mapping $g:\R^m \to \R^m$ such that $\varphi \circ g = g\circ T$. Hence 
\[ f\circ \Psi = \Psi \circ g \circ T \circ g^{-1}\]
and we may take $L = \Psi \circ g$.
\end{proof}

This mapping $L$ is called a Poincar\'{e} function or linearizer of $f$ at $x_0$, and is the mapping we will study.

\begin{remark}
Unlike the holomorphic case, a Poincar\'{e} function of a uqr mapping is not specified by the mapping and the repelling fixed point, since the infinitesimal space $\mathcal{D}f(x_0)$ may contain more than one mapping. We observe that by \cite[Lemma 4.4]{HMM} if one generalized derivative in an infinitesimal space is loxodromic repelling, then they all are.
\end{remark}

\begin{remark}
We may have chosen $T(x) = \lambda x$ for any $\lambda >1$. This will change the mapping $L$, but not its dynamical properties. Recall that $p'(z_0)$ is a fixed complex number, whereas in this quasiregular situation we are free to choose the multiplicative factor in the map we conjugate $f$ to about $x_0$.
\end{remark}

The first main result concerns the order of growth of such a linearizer.
The order of growth of a quasiregular mapping $f:\R^m \to \R^m$ is defined by
\[ \rho_f =  \limsup_{r \to \infty} (m-1) \frac{  \log \log M(r,f) }{\log r},\]
and the lower order is
\[ \lambda_f =  \liminf_{r \to \infty} (m-1) \frac{  \log \log M(r,f) }{\log r}.\]
Recall that the order of a linearizer of a polynomial $p$ of degree $d$ about $x_0$ is given by $\log d / \log |p'(x_0)|$, see \cite{Valiron}.

\begin{theorem}
\label{growth}
Let $f:\R^m \to \R^m$ be a $K$-uqr mapping of degree $d>K$ with repelling fixed point $x_0$ and let $L$ be a linearizer of $f$ about $x_0$ conjugating $f$ to $T$.
Then the order $\rho_L$ of $L$ satisfies
\[ \frac{ \log d - \log K }{ \log 2} \leq \rho_L \leq \frac{ \log d + \log K}{ \log 2}.\]
The same holds for the lower order.
\end{theorem}

\begin{remark}
In view of the fact we may have chosen $T(x) = \lambda x$ for any $\lambda >1$, Theorem \ref{growth} implies that we can construct linearizers of arbitrarily large or arbitrarily small positive order at a repelling fixed point of a uqr mapping.
\end{remark}

We say that $y\in\R^m$ is an omitted value of $L$ if there is no $x\in \R^m$ such that $L(x)=y$. By Rickman's Theorem \cite{Rickman}, $L$ can only omit finitely many values. We write $\mathcal{O}(L)$ for the set of omitted values of $L$. A point $x\in \R^m$ is called an exceptional value for a uqr mapping $f:\R^m \to \R^m$ if the backwards orbit $O^-(x) = \cup_{n\geq 1} f^{-n}(x)$ under $f$ is finite. We write $\mathcal{E}(f)$ for the set of exceptional values of $f$.
 If $f\circ L = L\circ T$, we relate the omitted values of $L$ to the exceptional values of $f$, compare \cite[Proposition 4.1]{MBP}.

\begin{theorem}
\label{omitted}
With the hypotheses of Theorem \ref{growth}, if $L$ is a linearizer of $f$ at $x_0$, we have $\mathcal{O}(L) = \mathcal{E}(f) \setminus \{ x_0 \}$.
\end{theorem}

For a quasiregular mapping $f:\R^m \to \R^m$ of transcendental type, the Julia set of $f$ is defined in \cite{BN} to be
\[ J(f) = \{ x\in \R^m : \operatorname{cap} \left ( \R^m \setminus O_f^+(U) \right ) = 0, \text{ for all open sets } U \ni x\},\]
where $O_f^+(U)$ denotes the forward orbit of $U$ under $f$.
We refer to \cite{BN} for the technical definition of capacity. The Julia set of a Poincar\'e linearizer has many properties in analogy with the Julia set of a transcendental entire function.

\begin{theorem}
\label{julialin}
With the hypotheses of Theorem \ref{growth},
we have
\begin{enumerate}[(i)]
\item $J(L) \subset \overline{O_L^-(x)}$ for all $x\in \R^m \setminus \mathcal{E}(L)$,
\item $J(L) \subset \overline{O_L^-(x)}$ for all $x\in J(L) \setminus \mathcal{E}(L)$,
\item $\R^m \setminus O_L^+(U) \subset \mathcal{E}(L)$ for every open set $U$ intersecting $J(L)$,
\item $J(L)$ is perfect,
\item $J(L^k) = J(L)$ for all $k\in \N$.
\item $J(L) = \partial A(L)$, recalling $A(L)$ is the fast escaping set.
\end{enumerate}
\end{theorem}

We next consider specific uqr mappings and their linearizers. Suppose that
\begin{equation}
\label{tame}
J(f) \text{ is a tame Cantor set.}
\end{equation}
Here, a tame Cantor set $E$ is one such that there is a homeomorphism $\psi : \R^m \to \R^m$ with $\psi(E)$ equal to a standard one thirds Cantor set contained in a line. Cantor sets which are not tame are called wild, for example Antoine's necklace. Every Cantor set in the plane is tame. See \cite{HR} for wild Cantor sets in the context of quasiregular mappings.

\begin{remark}
Condition \eqref{tame} may appear a restrictive condition, but there are plenty of quasiregular mappings satisfying it. 
In \cite{MP}, Martin and Peltonen show that every quasiregular mapping $h:S^m \to S^m$ can be decomposed as $h=f\circ \varphi$, where $\varphi$ is a quasiconformal map and $f$ is a uqr mapping for which $J(f)$ is a quasiconformally tame Cantor set. 
\end{remark}

Before we discuss dynamical properties of $L$, we first state the following result, which extends the known cases for when the Julia set of a uniformly quasiregular mapping is the closure of the repelling periodic points, compare with \cite{Siebert}.

\begin{theorem}
\label{dense}
Let $f:\R^m \to \R^m$ be a uniformly quasiregular mapping of polynomial type and suppose that $J(f)$ is a tame Cantor set. Then the repelling periodic points are dense in $J(f)$.
\end{theorem}

\begin{definition}
A set $E\subset \R^m$ is called a \emph{spider's web} if $E$ is connected and there exists a sequence of bounded topologically convex domains $G_n$ with $G_n \subset G_{n+1}$, for $n \in \N$, $\partial G_n \subset E$ and $\bigcup_{n \in \N} G_n = \R^m$.
\end{definition}

We recall that a topologically convex domain is one for which the only components of the complement are unbounded.
With this definition, $\R^m$ is itself a spider's web, but since every quasiregular mapping of transcendental type has infinitely many periodic points by \cite[Theorem 4.5]{Siebert2}, the fast escaping set can never be $\R^m$. We now state our final theorem.

\begin{theorem}
\label{mainthm}
Let $f:\R^m \to \R^m$ be a uqr mapping of polynomial type whose Julia set $J(f)$ is a tame Cantor set and let $x_0 \in J(f)$ be a repelling fixed point. If $L$ is a linearizer of $f$ at $x_0$, then $A(L)$  is a spider's web.
\end{theorem}

\begin{remark}
In \cite[Theorem 1.6]{BDF}, it is shown that if the minimum modulus of a quasiregular mapping $f$ of transcendental type is comparable to the maximum modulus on every annulus of the form $A(r,Cr)$ for some $C>1$, then $A(f)$ is a spider's web. This theorem cannot be used here since we do not have such estimates for linearizers. However, we prove Theorem \ref{mainthm} by showing that the minimum modulus is comparable to the maximum modulus on annuli of the form $A(r,r^{\mu})$ for some $\mu >1$.
\end{remark}

The rest of the paper is organized as follows. In section 2, we cover material from the iteration theory of quasiregular mappings, in section 3 we prove Theorem \ref{growth}, in section 4 we prove Theorem \ref{omitted}, in section 5 we prove Theorem \ref{julialin} in section 6 we prove Theorem \ref{dense} and finally in section 7 we prove Theorem \ref{mainthm}.

The author would like to thank Dan Nicks for helpful comments that improved the paper.

\section{Preliminaries}

Throughout we will write 
\[ B(x,R) = \{ y\in \R^m : |y-x| <R \} \]
for the ball of radius $R$ centred at $x$ and 
\[ A(R_1,R_2) = \{ y\in \R^m : R_1 < |y| < R_2 \} \]
for the annulus centred at $0$ with radii $R_1,R_2$.

\subsection{Quasiregular maps}

A mapping $f:E \rightarrow \R^{m}$ defined on a domain $E \subseteq \R^{m}$ is called quasiregular if $f$ belongs to the Sobolev space $W^{1}_{m, loc}(E)$ and there exists $K \in [1, \infty)$ such that 
\begin{equation}
\label{eq2.1}
\av f'(x) \av ^{m} \leq K J_{f}(x)
\end{equation}
almost everywhere in $E$. Here $J_{f}(x)$ denotes the Jacobian determinant of $f$ at $x \in E$. The smallest constant $K \geq 1$ for which (\ref{eq2.1}) holds is called the outer dilatation $K_{O}(f)$. If $f$ is quasiregular, then we also have
\begin{equation}
\label{eq2.2}
J_{f}(x) \leq K' \inf _{\av h \av =1} \av f'(x) h \av ^{n}
\end{equation}
almost everywhere in $E$ for some $K' \in[1, \infty)$. The smallest constant $K' \geq 1$ for which (\ref{eq2.2}) holds is called the inner dilatation $K_{I}(f)$. The dilatation $K=K(f)$ of $f$ is the larger of $K_{O}(f)$ and $K_{I}(f)$, and we then say that $f$ is $K$-quasiregular. Informally, a quasiregular mapping sends infinitesimal spheres to infinitesimal ellipsoids with bounded eccentricity.
A foundational result in the theory of quasiregular mappings is Rickman's Theorem, which states that a non-constant quasiregular mapping $f:\R^m\to \R^m$ can only omit $q(m,K)$ many values, where $q$ depends only on $m$ and $K$.
Quasiregular mappings are a generalization of analytic and meromorphic functions in the plane; see Rickman's monograph \cite{Rickman} for many more details. In particular, quasiregular mappings are open and discrete.

A quasiregular mapping $f:\R^m \to \R^m$ is said to be of polynomial type if $|f(x)| \to \infty$ as $|x| \to \infty$, whereas it is said to be of transcendental type if this limit does not exist and hence $f$ has an essential singularity at infinity. This is in direct analogy with polynomials and transcendental entire functions in the plane.

Denote by $i(x,f)$ the local index of $f$ at $x$.
The following result shows that quasiregular mappings are locally H\"older continuous.

\begin{theorem}[{\cite[Theorem III.4.7]{Rickman}}]
\label{rick2}
Let $f:E\to \R^n$ be quasiregular and non-constant, and let $x\in E$. Then there exist positive numbers $\rho, A,B$ such that for $y\in B(x,\rho)$,
\[ A|y-x|^{\nu} \leq |f(x)-f(y)| \leq B|y-x|^{\mu},\]
where $\nu = (K_O(f)i(x,f))^{1/(n-1)}$ and $\mu = (i(x,f)/K_I(f))^{1/(n-1)}$.
\end{theorem}

\subsection{Iteration of quasiregular mappings}

The composition of two quasiregular mappings is again a quasiregular mapping, but the dilatation typically increases. 
A quasiregular mapping $f$ is called uniformly $K$-quasiregular, or $K$-uqr, if the dilatation of each iterate $f^k$ of $f$ is bounded above by $K$. For uniformly quasiregular mappings, there are direct analogues of the Fatou set $F(f)$ and Julia set $J(f)$ for holomorphic mappings in the plane. 
In this case, the boundary of the escaping set coincides with the Julia set. 

The following was proved in \cite{FN2}.

\begin{theorem}[\cite{FN2}]
\label{unifperf}
Let $f:S^m \to S^m$ be uniformly quasiregular. Then $J(f)$ is $\alpha$-uniformly perfect, that is, if $R$ is any ring domain separating $J(f)$, then the conformal modulus of $R$ is at most $\alpha$, for some $\alpha >0$.
\end{theorem}

This result essentially says that any ring domain separating $J(f)$ cannot be too thick. We next prove a presumably well-known result on the growth of quasiregular mappings of polynomial type near infinity. 

\begin{lemma}
\label{holderinf}
Let $h:\R^m \to \R^m$ be a $K$-quasiregular mapping of polynomial type of degree $d>K$. Then there exist $R_0>0$ and positive constants $C_1,C_2$ such that
\[ C_1 ^{ q_j ( (d/K)^{1/(m-1)} ) } |x|^{ (d/K)^{j/(m-1)} }
\leq |h^j(x)|
\leq C_2 ^{ q_j ( (dK)^{1/(m-1)} ) } |x|^{ (dK)^{j/(m-1)} }, \]
for $|x| >R_0$, where $q_j$ is the polynomial $q_j(y) = y^{j-1} + y^{j-2} + \ldots + y +1$.
\end{lemma}

\begin{proof}
By the hypotheses, a neighbourhood of infinity is contained in $I(h)$, see \cite{FN}, and so infinity is an attracting fixed point of $h$. By the H\"{o}lder continuity for quasiregular mappings, Theorem \ref{rick2}, and conjugating by the M\"{o}bius map $x \mapsto x/|x|^2$, it is not hard to see that there exist $R_0>0$ such that $\{x\in \R^m : |x|>R_0\} \subset I(h)$ and positive constants $C_1,C_2$ such that
\[ C_1 |x|^{(d/K)^{1/(m-1)} }\leq |h(x)| \leq C_2 |x|^{(dK)^{1/(m-1)} },\]
for $|x|>R_0$.
The result then follows by induction.
\end{proof}

\subsection{The fast escaping set}

We summarize some of the results from \cite{BDF} on the fast escaping set.
Let $f:\R^m \to \R^m$ be a quasiregular mapping of transcendental type. 
\begin{defna}
The fast escaping set is
\[ A(f) = \{ x\in \R^m: \text{ there exists }P\in \N : |f^{n+P}(x)| \geq M^n(R,f), \text{ for all } n\in \N \},\]
where $R>0$ is any value such that $M^n(R,f) \to \infty$ as $n\to \infty$. 
\end{defna}
The fast escaping set does not depend on the particular value of $R$. For such values of $R$, the fast escaping set with respect to $R$ is 
\[ A_R(f) = \{ x \in \R^m : |f^n(x)| \geq M^{n}(R,f), n \in \N\}. \]

By \cite[Theorem 1.2]{BDF}, $A(f)$ is non-empty and every component of $A(f)$ is unbounded. We will use the following characterization of spider's webs for $A_R(f)$.

\begin{lemma}[Proposition 6.5, \cite{BDF}]
\label{char}
Let $f$ be a quasiregular mapping of transcendental type.
Then $A_R(f)$ is a spider's web if and only if there exists a sequence $G_n$ of bounded topoloigcally convex domains such that for all $n \geq 0$,
\[ B(0,M^n(R,f)) \subset G_n,\]
and $G_{n+1}$ is contained in a bounded component of $\R^m \setminus f(\partial G_n)$.
\end{lemma}

If $A_R(f)$ is a spider's web, then $A(f)$ is a spider's web, since $A_R(f) \subset A(f)$ and every component of $A(f)$ is unbounded.

\section{Order of growth}

\begin{proof}[Proof of Theorem \ref{growth}] 

Let $R_0$ be the constant from Lemma \ref{holderinf} and suppose that $M(r,L) \geq R_0$ for $r \geq r_1$. Fix $r_0 \geq r_1$ and let $r\in [r_0, 2r_0]$. Then by Lemma \ref{holderinf} and the fact that $M(r,L)$ is increasing in $r$,
\[ C_1 ^{ q_j ( (d/K)^{1/(m-1)} ) } M(r_0,L)^{ (d/K)^{j/(m-1)} }
\leq M(r,f^j \circ L)
\leq C_2 ^{ q_j ( (dK)^{1/(m-1)} ) } M(2r_0,L)^{ (dK)^{j/(m-1)} }, \]
where $q_j(y) = (y^j-1)/(y-1)$. Since $\log q_j(y) = j \log y +O(1)$ as $j \to \infty$, $f^j \circ L = L\circ T^j$ and $M(r,L \circ T^j) = M(2^jr, L)$ we have
\[ j \log \left ( \frac{d}{K} \right )^{1/(m-1)} -O(1) \leq \log \log M(2^jr, L) \leq j \log \left ( dK \right) ^{1/(m-1)} +O(1),\]
uniformly for $r\in[r_0,2r_0]$ as $j\to \infty$,
and hence
\[ \frac{ \log \left ( d/K \right )^{1/(m-1)} \log 2^j r }{\log 2} -O(1) \leq \log \log M(2^jr, L) \leq \frac{ \log \left ( dK \right) ^{1/(m-1)} \log 2^jr}{\log 2} +O(1),\]
which gives
\[ \frac{ \log \left ( d/K \right )^{1/(m-1)} }{\log 2} -o(1) \leq \frac{ \log \log M(2^jr, L)}{\log 2^j r} \leq \frac{ \log \left ( dK \right) ^{1/(m-1)}}{\log 2} +o(1)\]
uniformly for $r\in[r_0,2r_0]$ as $j\to \infty$.
These two inequalities imply the result for the order and the lower order. 

\end{proof}

\section{Omitted values}

\begin{proof}[Proof of Theorem \ref{omitted}]

Since $L(0) = x_0$, the point $x_0$ is never omitted. If $y\in \R^m \setminus \mathcal{E}(f)$, then the backward orbit $O^-(y)$ has infinitely many elements. Since $L$ omits at most $q(m,K)$ many values by Rickman's Theorem, $O^-(y)$ must intersect $L(\R^m)$. That is, there exists $k\in \N$ and $w\in \R^m$ with $L(w) \in f^{-k}(y)$. Therefore $y=f^k(L(w)) = L(2^kw)$ and so $y\notin \mathcal{O}(L)$. Hence $\mathcal{O}(L) \subset \mathcal{E}(f) \setminus \{x_0\}$.

Next, let $y\in \R^m \setminus \mathcal{O}(L)$. If $y=x_0$, there is nothing to prove, so suppose $y\neq x_0$. Then there exists $x\neq 0$ with $L(x) = y$. Hence $L(x/2^k) \in f^{-k}(y)$ by the iterated functional equation. Since $x\neq 0$ and $L$ is injective in a neighbourhood of $0$, the backwards orbit of $y$ has infinitely many elements.

\end{proof}

\section{The Julia set of $L$}

In this section we prove Theorem \ref{julialin}. As defined in \cite{BN}, a quasiregular mapping $f:\R^m\to \R^m$ has the {\it pits effect} if there exists $N\in \N$ such that for all $c>1$ and all $\epsilon >0$, there exists $R_0$ such that if $R>R_0$, then the set
\[ \{ x\in \R^m  : R\leq |x| \leq cR, |f(x)|\leq 1 \}\]
can be covered by $N$ balls of radius $\epsilon R$, that is, the set where $f$ is small is not too large. By \cite[Theorem 1.8]{BN}, if a quasiregular mapping of transcendental type does not have the pits effect, then the Julia set has the properties given in the statement of Theorem \ref{julialin} $(i)-(v)$. Hence parts $(i)-(v)$ of the theorem are proved by the following lemma.

\begin{lemma}
\label{julialemma}
Let $L:\R^m \to \R^m$ be a Poincar\'e linearizer. Then $L$ does not have the pits effect.
\end{lemma}

\begin{proof}
As noted after Definition 1.2 in \cite{BN}, it suffices to show that there is a sequence $(x_k)_{k=1}^{\infty}$ tending to $\infty$ such that $\limsup_{k\to \infty} |x_{k+1}|/|x_k|  < \infty$ and 
\begin{equation}
\label{juliaeq1}
|L(x_k)| \leq C 
\end{equation}
for all $k\in \N$ and some positive constant $C$.

Recall from Theorem \ref{unifperf}
that the Julia set of a uqr mapping $f$ is $\alpha$-uniformly perfect, that is, if $E$ is any ring domain in $F(f)$ separating points of $J(f)$ then the conformal modulus of $E$ is at most $\alpha$ for some $\alpha$ depending on $f$.

If $L$ is $K$-quasiregular, then suppose that $D>0$ is chosen so that $\frac{1}{2K} \log D > \alpha$. 
Next, since $0$ is not a branch point of $L$, there exists a neighbourhood $U$ of $0$ 
such that $L \av_{U}$ is injective and we may assume that 
$L(U) \cap J(f) \neq J(f)$.

Consider the annulus $A=A(r,Dr)$, and choose $j \in \N$ large enough such that $T^{-j}(A) = A(2^{-j}r, 2^{-j}Dr) \subset U$. Then $L(T^{-j}(A))$ is a ring domain separating points of $J(f)$ with modulus greater than $\alpha$ and hence must intersect $J(f)$. Therefore, with $C = \max _{x \in J(f)} |x|$, there exists $y \in L(U) \cap J(f)$ which satisfies
\[ |f^j(y)| \leq C,\]
for all $j \in \N$. There exists $y' \in U$ such that $L(y') = y$ and $T^j(y') \in A$. Hence $T^j(y')$ is a point in $A(r,Dr)$ with $|L(T^j(y'))| = |f^j(L(y'))| \leq C$.

This argument shows that we can find a sequence of points $x_k \in A(D^{k-1}r,D^kr)$ for $k\in \N$ which satisfy \eqref{juliaeq1} with this $C$ and such that $|x_{k+1}|/|x_k| \leq D^2$. This proves the lemma.
\end{proof}

For part $(vi)$, it is shown in \cite{BFN} that for a quasiregular mapping of transcendental type of positive lower order, the Julia set and the boundary of the fast escaping set agree. Hence part $(vi)$ follows from Theorem \ref{growth}.

\section{Density of repelling periodic points}

\begin{proof}[Proof of Theorem \ref{dense}]

By a result contained in Siebert's thesis \cite{Siebert}, and see also \cite[Theorem 4.1]{B2} and the discussion preceding it, the periodic points are dense in the Julia set of a uniformly quasiregular mapping. Suppose that $f:\R^m \to \R^m$ is a uqr mapping of polynomial type and that $J(f)$ is a tame Cantor set. Let $x_0$ be a periodic point of period $p$, and write $F=f^p$ so that $x_0$ is a fixed point of $F$.

Let $U_{\delta}$ be a neighbourhood of $x_0$ such that the diameter of $U_{\delta}$ is at most $\delta$ and $\partial U_{\delta} \subset I(f)$. That such neighbourhoods exist follows from the tameness of $J(f)$ and the fact that every bounded component of $F(f)$ is simply connected since $J(f) = \partial I(f)$.

Let $\delta_n \to 0$ be such that $\overline{U_{\delta_{n+1}}} \subset U_{\delta_n}$. We must have that $F$ is injective on a neighbourhood of $x_0$ since otherwise $x_0$ is an attracting fixed point by Theorem \ref{rick2} (see also \cite[equation (7)]{HMM}). This is impossible since there are escaping points arbitrarily close to $x_0$. Find $\epsilon >0$ so that $F$ is injective on $B(x_0,\epsilon)$. We may assume $\delta_n <\epsilon$ for all $n$.

Choose $n$ large enough so that $U_{\delta_n}$ is a very small neighbourhood of $x_0$ and so that there exists $k\in \N$ with $\overline{U_{\delta_n}} \subset F^k(U_{\delta_n}) \subset B(x_0,\epsilon)$. That we can find such a $k$ follows since $\partial U_{\delta_n} \subset I(F)$ (recall $I(F)=I(f)$) and as long as $F^j(U_{\delta_n}) \subset B(x_0,\epsilon)$, the mapping $F^j |_{U_{\delta_n}}$ is $K$-quasiconformal, and then $\max _{x\in \delta_n} |F^j(x)| / \min_{x\in \delta_n} |F^j(x)| <C$ for some $C$ depending on $K$.

Hence by the topological definition of fixed points, see \cite[p.90]{HMM}, $x_0$ is a repelling fixed point of $F^k$. Finally, by \cite[Proposition 4.6]{HMM}, $x_0$ is a repelling periodic point of $f$.

\end{proof}

\section{The fast escaping set of $L$}

Suppose that $f:\R^m \to \R^m$ is a $K$-uqr mapping of polynomial type, that $J(f)$ is a tame Cantor set and that $x_0$ is a repelling periodic point guaranteed by Theorem \ref{dense}. We may assume that the degree $d$ of $f$ is greater than $K$, otherwise consider an iterate $f^n$ of $f$ so that $nd >K$. Note that $J(f^n) = J(f)$ by \cite[Corollary 3.3]{HMM}.
Let $L$ be a linearizer of $f$ at $x_0$, conjugating $f$ to $T(x) =2x$. Throughout this section, we will write
\[ \beta = \left ( \frac{d}{K} \right ) ^{1/(m-1)}.\]

\subsection{H\"{o}lder continuity and growth estimates}

We next use the H\"{o}lder estimate for the iterates of $f$ from Lemma \ref{holderinf} to obtain growth estimates for the linearizer $L$.

\begin{lemma}
\label{reggrowth}
There exists $R_1>0$ such that
\[ \prod _{i=1}^{n-1} \left ( \left ( \frac{d}{K} \right ) ^{1/(m-1)} + \frac{ \log C_1}{\log M(2^ir,L) } \right ) \leq
\frac { \log M(2^nr,L)}{\log M(r,L) } 
\leq \prod _{i=1}^{n-1} \left (  (dK ) ^{1/(m-1)} + \frac{ \log C_2}{\log M(2^ir,L) } \right ),\]
for $r>R_1$, where the constants $C_1,C_2$ are from Lemma \ref{holderinf}.
\end{lemma}

\begin{proof}

Denote by $S_r = \{ x\in \R^m : |x| = r \}$ the sphere of radius $r$ centred at $0$. Let $r$ be large and $y \in S_r$ such that $|L(y)| \geq |L(x)|$ for all $x \in S_r$. Let $w = L(y)$ so that $|w| = M(r,L)$. Then by the functional equation for $L$ and  Lemma \ref{holderinf},
\begin{align*}
\log M(2r,L) & = \log M(r,f \circ L) \\
&= \log M( L(S_r),f )\\
& \geq \log |f(w)| \\
& \geq \log C_1 +(d/K)^{1/(m-1)} \log |w| \\
&= \log C_1 + (d/K)^{1/(m-1)} \log M(r,L).
\end{align*}
Also,
\begin{align*}
\log M(2r,L) & \leq \log M(M(r,L),f) \\
& \leq \log C_2 + (dK)^{1/(m-1)} \log M(r,L),
\end{align*}
and so we have
\[ \left( \frac{d}{K} \right ) ^{1/(m-1)} + \frac{\log C_1}{\log M(r,L)} \leq \frac{ \log M(2r,L)}{\log M(r,L) } \leq 
(dK)^{1/(m-1)}+ \frac{\log C_2}{\log M(r,L)}.\]
By induction, we have the result.
\end{proof}

\begin{lemma}
\label{reggrowth2}
Let $\mu >1$. There exist $R_2>0$ and $N\in \N$ such that for all $R>R_2$, the sequence defined by
\begin{equation}\label{eq:reg1} 
r_n = 2^n M^n(R,L)
\end{equation}
satisfies
\[ M(r_n,L) > r_{n+1}^{\mu},\]
for $n\geq N$.
\end{lemma}

\begin{proof}
Assume $R$ is large.
With the sequence $r_n$ defined by \eqref{eq:reg1}, applying Lemma \ref{reggrowth} with $r = M^n(R,L)$ yields
\begin{align*}
\log M(r_n,L) & = \log M(2^nM^n(R,L),L) \\
& \geq \prod _{i=0}^{n-1} \left ( \left ( \frac{d}{K} \right )^{1/(m-1)} + \frac {\log C_1 }{\log M(2^i M^n (R,L) , L) } \right ) \cdot \log M( M^n(R,L) , L) \\
&\geq \left ( \beta - \frac{ \log C_1}{\log R} \right ) ^n \log M^{n+1}(R,L).
\end{align*}
Now,
\begin{align*}
\log r_{n+1}^{\mu} & = \mu \log ( 2^{n+1} M^{n+1}(R,L) ) \\
&= \mu (n+1) \log 2 + \mu \log M^{n+1}(R,L).
\end{align*}
Hence the result is true if
\[ \log M^{n+1}(R,L) \left (  \left ( \beta - \frac{\log C_1}{\log R} \right )^n - \mu \right ) > \mu (n+1) \log 2.\]
This is so if we choose $R$ large enough and $n$ large enough so that $\beta^n > \mu$.
\end{proof}

We next prove a lemma on the growth of the minimum modulus of $L$.

\begin{lemma}
\label{minmodgrowth}
Suppose that $f:\R^m \to \R^m$ is $K$-uqr of degree $d>K$, $J(f)$ is a tame Cantor set, $x_0$ is a repelling fixed point and $L$ is the corresponding linearizer. 
Let 
\[ \mu > \frac{ \log d + \log K}{\log d- \log K}.\]
There exists $R_3>0$ such that for $r>R_3$ there is a continuum $\Gamma^r$ separating $S_r$ and $S_{r^{\mu}}$ such that
\[ m(\Gamma^r,L) > M(r,L).\]
\end{lemma}

\begin{proof}
There exists a neighbourhood $U$ of $0$ such that $L |_{U}$ is injective. Let $\delta >0 $ be small enough that $B(x_0,\delta) \subset L(U)$. Since $J(f)$ is a tame Cantor set, there exists a topologically convex neighbourhood $V$ of $x_0$ such that $\gamma_\delta := \partial V \subset B(x_0,\delta) \cap I(f)$.
Let $\Gamma _{\delta} = L^{-1}(\gamma_{\delta}) \cap U$. Then $\Gamma_{\delta}$ is a continuum which separates $0$ from infinity.

Suppose that $\Gamma_{\delta} \subset A(s,t)$ and without loss of generality, we may assume that $t<1$. Let $r$ be large and find $\ell_1,\ell_2 \in \N$ such that
\[ 2^{\ell_1 -1} \leq r < 2^{\ell_1},\]
and 
\[ t \cdot 2^{\ell_2} \leq r^{\mu}  < t \cdot 2^{\ell_2+1}.\]
This pair of inequalities implies that
\begin{equation}
\label{minmodeq0} 
\mu \ell_1 - D_1 < \ell_2 < \mu \ell_1 -D_2,
\end{equation}
where $D_1 = \log t / \log 2 + \mu +1$ and $D_2 = \log t / \log 2$.

Next, since $\gamma_{\delta} \subset I(f)$ is compact, find $j \in \N$ minimal such that $f^{j}(\gamma_{\delta}) \subset \{ |x| > R_0\}$ and so we can apply Lemma \ref{holderinf}. Define $\Gamma^r = \{ x \in \R^m : 2^{-\ell_2}\cdot x \in \Gamma_{\delta} \}$. Then $\Gamma^r$ separates $S_r$ and $S_{r^{\mu}}$.

We first estimate the minimum modulus on $\Gamma^r$:
\begin{align*}
\log m(\Gamma^r,L) & = \log m( \Gamma^r, f^{\ell_2} \circ L \circ T^{-\ell_2} ) \\
& = \log m( \Gamma_{\delta}, f^{\ell_2} \circ L ) \\
& = \log m( \gamma_{\delta}, f^{\ell_2} ) \\
& \geq \log m(R_0, f^{\ell_2 - j} ) \\
& \geq q_{\ell_2 - j} ( (d/K)^{1/(m-1)} ) \log C_1 + (d/K)^{(\ell_2 - j)/(m-1)} \log R_0.
\end{align*}
Next,
\begin{align*}
\log M(r,L) & = \log M(L( S_{2^{-\ell_1}r} ), f^{\ell_1} ) \\
& \leq \log M( R_0, f^{\ell_1} ) \\
& \leq q_{\ell_1}( (dK)^{1/(m-1)} ) \log C_2 + (dK)^{\ell_1 / (m-1)} \log R_0 .
\end{align*}
Since $y^{j-1} \leq q_j(y) \leq y^{j}$, we obtain
\begin{equation}
\label{minmodeq1}
\log m(\Gamma^r, L) \geq \left ( \frac{d}{K} \right ) ^{(\ell_2 - j -1)/(m-1) } \log C_1 + \left ( \frac{d}{K} \right ) ^{(\ell_2 - j)/(m-1) } \log R_0
\end{equation}
and 
\begin{equation}
\label{minmodeq2}
\log M(r,L) \leq (dK)^{\ell_1 / (m-1) } \log (C_2R_0).
\end{equation}
Using \eqref{minmodeq0} and \eqref{minmodeq1}, we obtain
\begin{equation}
\label{minmodeq3}
\log m(\Gamma^r, L) \geq \left ( \frac{d}{K} \right ) ^{ ( \mu\ell_1 - D_1 - j -1)/(m-1)} \log C_1 + \left ( \frac{d}{K} \right ) ^{ (\mu \ell_1 - D_1 - j )/(m-1) } \log R_0.
\end{equation}
Recall that $\beta = (d/K)^{1/(m-1)} >1$.
Therefore, to obtain $\log m(\Gamma^r, L) \geq \log M(r,L)$, by \eqref{minmodeq2} and \eqref{minmodeq3} and rearranging, it suffices to show that
\[ \beta^{ \mu \ell_1 - D_1 - j - \ell_1 - 1}( \log C_1 + \beta \log R_0) > \beta^{\ell_1} K^{2\ell_1/(m-1)} \log (C_2R_0).\]
Recalling that $\ell_1$ depends on $r$ and writing $C$ for a constant which does not depend on $r$, by taking logarithms this can be written as
\[ \ell_1 \left ( (\mu -1 )\log \beta - \frac{2 \log K}{m-1} \right ) \geq C.\]
This is satisfied for large enough $r$, that is for large enough $\ell_1$, if 
\[ (\mu-1) \log \beta > \frac{2 \log K}{m-1} ,\]
that is, if
\[ \mu > 1 + \frac{ 2\log K}{(m-1)\log \beta } = \frac{\log d + \log K}{\log d - \log K}.\]
\end{proof}

\subsection{$A(L)$ is a spider's web}

We will use the Lemma \ref{char} characterization of spider's webs.

Recalling Lemmas \ref{holderinf}, \ref{reggrowth}, \ref{reggrowth2} and \ref{minmodgrowth}, let $R>\max \{ R_0,R_1,R_2,R_3 \}$ and for $n \in \N$, let $r_n = 2^n M^{n}(R,L)$ and let $\mu > \frac{ \log d + \log K}{\log d- \log K}$. By Lemma \ref{minmodgrowth}, there is a continuum $\Gamma^{r_n}$ separating $S_{r_n}$ and $S_{r_{n}^{\mu}}$ such that
\[ m(\Gamma^{r_n}, L) >M(r_n,L).\]
We define $G_n$ to be the interior of $\Gamma^{r_n}$. Then by construction, every $G_n$ is a bounded topologically convex domain with
\[ G_n \supset \{ x \in \R^m : |x| <r_n \} \supset \{ x \in \R^m : |x| <M^n(R,L) \}.\]
Further, it follows from Lemma \ref{reggrowth2} that
\[ m(\partial G_n, L) = m(\Gamma^{r_n},L) > M(r_n,L) > r_{n+1}^{\mu} > \max _{x \in \partial G_{n+1} } |x|,\]
and hence $G_{n+1}$ is contained in a bounded component of $\R^m \setminus L(\partial G_n )$ and we have fulfilled the conditions of Lemma \ref{char} for $A(L)$ to be a spider's web.


\begin{thebibliography}{widest-label}



\bibitem{B2}
W. Bergweiler,
Iteration of quasiregular mappings, 
{\it Comput. Methods Funct. Theory}, {\bf 10} (2010), no. 2, 455-481.

\bibitem{B3}
W. Bergweiler,
Fatou-Julia theory for non-uniformly quasiregular maps, {\it Ergodic Theory Dynam. Systems},
{\bf 33} (2013), 1-23.

\bibitem{BDF}
W. Bergweiler, D. Drasin, A. Fletcher,
The fast escaping set for quasiregular mappings,
submitted.

\bibitem{BFLM}
W. Bergweiler, A. Fletcher, J. K. Langley, J. Meyer,
The escaping set of a quasiregular mapping,
{\it Proc. Amer. Math. Soc.}, {\bf 137} (2009) 641-651.

\bibitem{BFN}
W. Bergweiler, A. Fletcher, D. A. Nicks,
The Julia set and the fast escaping set of a quasiregular mapping,
in preparation.

\bibitem{BH}
W. Bergweiler, A. Hinkkanen, 
On semiconjugation of entire functions,
{\it Math. Proc. Cambridge Philos. Soc.}, {\bf 126} (1999), no. 3, 565-574. 

\bibitem{BN}
W. Bergweiler, D. A. Nicks,
Foundations for an iteration theory of entire quasiregular maps,
to appear in {\it Israel J. Math.}


\bibitem{E}
A. E. Eremenko, 
On the iteration of entire functions, 
{\it Dynamical systems and ergodic theory}, Banach Center Publications 23, Polish Scientific Publishers, Warsaw, 1989, 339-345.

\bibitem{FN} 
A. Fletcher, D. A. Nicks,
Quasiregular dynamics on the $n$-sphere,
{\it Ergodic Theory and Dynamical Systems}, {\bf 31} (2011), 23-31.

\bibitem{FN2} A. Fletcher, D. A. Nicks, Julia sets of uniformly quasiregular mappings are uniformly perfect, {\it Math. Proc. Cam. Phil. Soc.}, {\bf 151}, no.3 (2011), 541-550.

\bibitem{GMRV} V. Gutlyanskii, O. Martio, V. Ryazanov, M. Vuorinen,
Infinitesimal geometry of quasiregular mappings,
{\it Ann. Acad. Sci. Fenn.}, {\bf 25} (2000), 101-130.

\bibitem{HM} A. Hinkkanen, G. J. Martin, 
Attractors in quasiregular semigroups,
{\it XVIth Rolf Nevanlinna Colloquium (Joensuu, 1995)}, de Gruyter, Berlin (1996), 135-141.

\bibitem{HMM} A. Hinkkanen, G. J. Martin, V. Mayer, Local dynamics of uniformly quasiregular mappings, {\it Math. Scand.}, {\bf 95} (2004), no. 1, 80-100. 

\bibitem{HR} J. Heinonen, S. Rickman, Quasiregular maps $S^3\to S^3$ with wild branch sets, {\it Topology}, {\bf 37} (1998), no. 1, 1-24. 

\bibitem{IM} T. Iwaniec, G. J. Martin, Quasiregular semigroups, {\it Ann. Acad. Sci. Fenn.}, {\bf 21} (1996), 241-254.

\bibitem{MP} G. J. Martin, K. Peltonen, Stoilow factorization for quasiregular mappings in all dimensions, {\it Proc. Amer. Math. Soc.} {\bf 138} (2010), no. 1, 147-151. 


\bibitem{MBP} H. Mihaljevic-Brandt, J. Peter, Poincar\'{e} functions with spiders' webs,
{\it Proc. Amer. Math. Soc.} {\bf 140} (2012), 3193-3205. 

\bibitem{Milnor} J.Milnor,
{\it Dynamics in one complex variable}, Third edition, Annals of
Mathematics Studies, {\bf 160}, Princeton University Press, Princeton, NJ, 2006.

\bibitem{Rickman} S. Rickman, 
{\it Quasiregular mappings}, Ergebnisse der Mathematik und ihrer
Grenzgebiete 26, Springer, 1993.


\bibitem{RS}
P. Rippon, G. Stallard, 
Fast escaping points of entire functions,
{\it Proc. London Math. Soc.}, {\bf 105} (2012), 787-820.

\bibitem{RS2012}
P. Rippon, G. Stallard,
A sharp growth condition for a fast escaping spider's web,
{\it Adv. Math.}, {\bf 244} (2013), 337-353.

\bibitem{Siebert} H. Siebert, {\it Fixpunkte und normale Familien quasiregul\"{a}rer Abbildungen}, Dissertation, CAU Kiel, 2004.

\bibitem{Siebert2} H. Siebert, Fixed points and normal families of quasiregular mappings, 
{\it Journal d'Analyse Math\'{e}matique}, {\bf 98}, no. 1 (2006), 145-168 .

\bibitem{Valiron} G. Valiron, Fonctions Analytiques, Presses de l'Universit\'{e}, 1954.


\end{thebibliography}
\end{document}